\documentclass[a4paper,draft,reqno,12pt]{amsart}
\usepackage[utf8]{inputenc}
\usepackage[T1,T2A]{fontenc}
\usepackage[english]{babel}
\usepackage{amsmath}
\usepackage{amssymb}
\usepackage{amscd}
\usepackage{amsthm}
\usepackage{euscript}
\usepackage[all]{xy}
\newtheorem{proposition}{Proposition}

\newtheorem{lemma}{Lemma}
\newtheorem{theorem}{Theorem}

\theoremstyle{definition}

\newtheorem{example}{Example}
\theoremstyle{remark}
\newtheorem {remark}{Remark}

\DeclareMathOperator{\Aut}{Aut}
\DeclareMathOperator{\Bir}{Bir}

\DeclareMathOperator{\PGL}{PGL}

\def\GG{{\mathbb G}}

\def\CC{{\mathbb C}}
\def\KK{{\mathbb K}}

\def\PP{{\mathbb P}}
\def\AA{{\mathbb A}}

\def\mm{\mathfrak{m}}

\renewcommand{\phi}{\varphi}
\renewcommand{\ge}{\geqslant}
\renewcommand{\le}{\leqslant}

\begin{document}
\date{}
\title[On conjugacy of additive actions]{On conjugacy of additive actions in the affine Cremona group}
\author{Ivan Arzhantsev}
\address{Faculty of Computer Science, HSE University, Pokrovsky Boulevard 11, Moscow, 109028 Russia}
\email{arjantsev@hse.ru}
\thanks{Supported by the Russian Science Foundation grant 23-21-00472}

\maketitle
\begin{abstract}
An additive action on an irreducible algebraic variety $X$ is an effective action $\GG_a^n\times X\to X$ with an open orbit of the vector group $\GG_a^n$. Any two additive actions on $X$ are conjugate by a birational automorphism of $X$. We prove that, if $X$ is the projective space, the conjugating element can be chosen in the affine Cremona group and it is given by so-called basic polynomials of the corresponding local algebra.  
\end{abstract} 


\section{Introduction}
\label{sec1}

Let $\KK$ be an algebraically closed field of characteristic zero and $\GG_a=(\KK,+)$ be the additive group of the field $\KK$. It is well known that any commutative unipotent linear algebraic group over $\KK$ is isomorphic to the vector group $\GG_a^n$. 

An additive action on an irreducible algebraic variety $X$ is an effective action $\GG_a^n\times X\to X$ with an open orbit. In the study of additive actions it is natural to look for analogies with  toric geometry. Namely, we consider the multiplicative group $\GG_m$ of the field $\KK$ and an algebraic torus $T=\GG_m^n$. Let us recall that a normal variety $X$ is toric if it admits a regular action $T\times X\to X$ with an open orbit. 

It is known that an effective torus action with an open orbit on a toric variety $X$ is unique up to automorphism of $X$. Namely, if $X$ is a complete toric variety, it is shown in~\cite{De-1} that the automorphism group $\Aut(X)$ is a linear algebraic group, and the claim follows from the fact that any two maximal tori in a linear algebraic group are conjugate. For an arbitrary toric variety more specific arguments are needed; see~\cite[Theorem~4.1]{Be}. 

At the same time, many varieties $X$ admit several non-equivalent additive actions. For the first time, this effect was observed by Hassett and Tschinkel in the case of projective spaces. More precisely, in~\cite{HT} a correspondence between commutative associative unital local algebras $A$ with $\dim A=n+1$ and additive actions on $\PP^n$ was established; see also~\cite{KL}. Using a classification of local algebras of small dimensions, one can conclude that starting from $n=2$ an additive action on $\PP^n$ is not unique, and starting from $n=6$ there are infinitely many
equivalence classes of such actions.    

During last decades, many results on additive actions on complete algebraic varieties were obtained. In particular, results on uniqueness of an additive action on non-degenerate quadrics and, more generally, arbitrary non-degenerate projective hypersurfaces, flag varieties, and other types of varieties can be found in~\cite{AZ,De,DKM,Dz-2,FH,Sh}. On the other hand, examples of non-equivalent additive actions on complete varieties are given in~\cite{ABZ,AP,AR,Bel,Dz-1,Liu,Sha-2}. For a recent survey of these results, see~\cite{AZ}. 

Returning to linear algebraic groups, it is well known that in any such group all maximal unipotent subgroups are conjugate.  At the same time, every additive action on $\PP^n$ gives rise to
a separate conjugacy class of maximal commutative unipotent subgroups in ${\PGL(n+1)}$. Moreover, one can construct commutative unipotent subgroups in $\PGL(n+1)$ whose dimension is greater than~$n$. 

\smallskip

This note originates from attempts to analyze these effects. It is easy to see that any two additive actions on a variety $X$ are birationally equivalent or, in other words, they are conjugate in the group of birational automorphisms of $X$.  More precisely, if two additive actions on $X$ share the same open orbit $U$, they are conjugate in the automorphism group $\Aut(U)$; see Proposition~\ref{prop1}. Since the orbit $U$ is isomorphic to an affine space, the group $\Aut(U)$ is the affine Cremona group. 

The aim of this note is to show that in the case of the projective space $\PP^n$ an element that conjugates a given additive action to the standard one has a remarkable interpretation in terms of the Hassett-Tschinkel correspondence. Namely, any local algebra $A$ with $\dim A=n+1$ defines a so-called basic subspace $V$ in the polynomial algebra $\KK[x_1,\ldots,x_n]$. This subspace can be defined in terms of exponents of elements in $A$; see Section~\ref{sec3} for details. The subspace $V$ is $(n+1)$-dimensional, invariant under all translations
${(x_1+c_1,\ldots,x_n+c_n)}$, $c_i\in\KK$, and generates the algebra $\KK[x_1,\ldots,x_n]$. Moreover, equivalence classes of such subspaces are in natural bijection with isomorphism classes of $(n+1)$-dimensional local algebras; see~\cite[Theorem~2.14]{HT} and \cite[Theorem~1.48]{AZ}. 

We prove that a basis $1, f_1,\ldots,f_n$ of the basic subspace $V$ defines an automorphism $x_1\mapsto f_1,\ldots,x_n\mapsto f_n$ from the affine Cremona group that conjugates the additive action corresponding to the algebra $A$ to the standard additive action; see Theorem~\ref{tmain}.  This result shows that one conjugacy class of $\GG_a^n$-subgroups in the affine Cremona group intersects the subgroups of the group $\PGL(n+1)$ in several $\PGL(n+1)$-conjugacy classes. Such an intersection is called in~\cite{Bl} the trace on $\PGL(n+1)$ of a conjugacy class in $\Aut(U)$; see~\cite{Bl} and references therein for results on an interplay of conjugacy classes of elements in the groups we are dealing with. Also Theorem~\ref{tmain} provides a new characterization of basic subspaces in the polynomial algebra.

\smallskip

The author is grateful to Ivan Beldiev and Yulia Zaitseva for useful comments.


\section{General results on conjugation}
\label{sec2}

We denote by $\Aut(X)$ the group of regular automorphisms and by $\Bir(X)$ the group of birational automorphisms of an irreducible algebraic variety $X$. Let $\Aut(n,\KK)$ be the group of $\KK$-automorphisms of the polynomial algebra $\KK[x_1,\ldots,x_n]$ and $\Bir(n,\KK)$ be the group of $\KK$-automorphisms of its field of fractions $\KK(x_1,\ldots,x_n)$. The group $\Bir(n,\KK)$ is the Cremona group; it can be identified with $\Bir(\AA^n)$, with $\Bir(\PP^n)$, and with $\Bir(X)$ for any rational variety $X$. In particular, the group $\Bir(n,\KK)$  coincides with $\Bir(X)$ for any irreducible variety $X$ admitting an additive action. 

The group $\Aut(n,\KK)$ is called the affine Cremona group; it can be identified with the group $\Aut(\AA^n)$. The group $\Aut(\PP^n)$ is much smaller, it is isomorphic to $\PGL(n+1)$. 

\begin{proposition} \label{prop1}
Let $X$ be an irreducible algebraic variety. Then
\begin{enumerate}
\item
any two additive actions on $X$ are conjugate in the group $\Bir(X)$;
\item
if two additive actions have the same open orbit $U$ in $X$, then these actions are conjugate in the affine Cremona group $\Aut(U)\subseteq\Bir(X)$. 
\end{enumerate}
\end{proposition} 

\begin{proof}
Let $U_1$ and $U_2$ be the open orbits of two given additive actions on $X$. Then an isomorphism of varieties $\psi\colon U_1\to U_2$ defines a birational automorphism of $X$, and conjugating by this automorphism, we may assume that the two additive actions have the same open orbit $U$. 

So it suffices to prove claim~(2). By definition of an additive action, the transitive action $\GG_a^n\times U\to U$ is effective, so it is free. Fixing a point $x_0\in U$, the orbit map
$\GG_a^n\to U$, $g\mapsto g\cdot x_0$ defines an isomorphism $\GG_a^n\cong U$. Having two such isomorphisms corresponding to the two given additive actions, we obtain two structures of the affine space on the open subset $U$. The passage from one such structure to another gives rise to an element of the group $\Aut(U)$ that conjugates the first additive action to the second one. 
\end{proof} 

\begin{remark}
Proposition~\ref{prop1}~(1) reflects a more general and rather obvious fact: any two actions of an algebraic group $G$ with an open orbit on an irreducible variety $X$ which have the same generic stabilizers are birationally equivalent. 
\end{remark}

\begin{remark}
It follows from the description of additive actions on projective spaces that in the case $X=\PP^n$ the map $\psi$ from the proof above is an automorphism of the variety $X$. We do not know whether it is the case in general. 
\end{remark} 

We finish this section with one more general observation. Let $\alpha\colon\GG_a^n\times X\to X$ be an additive action on an irreducible variety $X$. Denote by 
$\beta\colon\GG_a^n\to\Aut(X)$ the corresponding injective homomorphism and consider the subgroup $H=\beta(\GG_a^n)$ in $\Aut(X)$.  

\begin{proposition} 
The subgroup $H$ is a maximal connected commutative algebraic subgroup both in $\Aut(X)$ and $\Bir(X)$. 
\end{proposition} 

\begin{proof}
It suffices to prove the claim for the group $\Bir(X)$. Assume that $G$ is a connected commutative algebraic subgroup in $\Bir(X)$ that contains the subgroup $H$ properly.
Then $\dim G>n$. On the other hand, by the Regularization Theorem of Andr\'e Weil~\cite{We} (see~\cite{Kr} for a modern proof) there exist a variety $Y$ with a regular action of $G$ and a $G$-equivariant birational map $X\to Y$. Since the group $H$ acts on $X$ with an open orbit, the field of rational invariants $\KK(X)^H$ coincides with $\KK$. But the group $G$ contains $H$, so we have $\KK(X)^G=\KK(Y)^G=\KK$. We conclude that the effective regular action $G\times Y\to Y$ has an open orbit. This contradicts the condition $\dim G>n=\dim Y$. 
\end{proof} 

At the same time, for $n\ge 2$ the groups $\Aut(n,\KK)$ and $\Bir(n,\KK)$ contain commutative unipotent algebraic subgroups of arbitrary dimensions: one can take all transformations of the form
$$
(x_1+f(x_2),x_2,\ldots,x_n),
$$ 
where $f$ runs through all polynomials of degree less than $d$ for some positive integer $d$. 

\section{Hassett-Tschinkel correspondence}
\label{sec3}

In~\cite{HT}, a bijective correspondence between additive actions on the projective space $\PP^n$ and commutative associative unital local algebras $A$ with $\dim A=n+1$ is established; see also~\cite{KL}. Let us fix a basis $s_1,\ldots,s_n$ in the maximal ideal $\mm$ of the algebra $A$. Under this correspondence, the space $\PP^n$ is identified with the projectivization $\PP(A)$, and an element $y=(y_1,\ldots,y_n)\in\GG_a^n$ acts on $\PP(A)$ by multiplication by an element $\exp(y_1s_1+\ldots+y_ns_n)$.

A classification of local algebras up to dimension~6 is given in~\cite{HT} with a reference to~\cite{ST}; see also~\cite[Table~1]{AZ}. This classification was obtained independently in~\cite{Ma}; one more approach can be found in~\cite{Po}.  As a result of the classification, we have the following number of isomorphism classes of local algebras of dimension~$n+1$: 
\[
\begin{array}{c|c|c|c|c|c|c|c}
n +1 & 1 & 2 & 3 & 4 & 5 & 6 & \ge 7\\\hline
& 1 & 1 & 2 & 4 & 9 & 25 & \infty
\end{array}
\eqno (1)
\]
With any local algebra $A$ one associates polynomials $f_1,\ldots,f_n\in\KK[x_1,\ldots,x_n]$ defined be the formula
$$
\exp(x_1s_1+\ldots+x_ns_n)=1+f_1(x)s_1+\ldots+f_n(x)s_n.
$$
Choosing a basis $s_1,\ldots,s_n$ in $\mm$ compatible with the filtration 
$$
\mm\supseteq\mm^2\supseteq\ldots\supseteq\mm^{d-1}\supseteq\mm^d=0,
$$
we observe that each polynomial $f_i$ has the form $x_i+h_i(x_1,\ldots,x_{i-1})$. So the endomorphism of the algebra $\KK[x_1,\ldots,x_n]$ given by $x_1\to f_1(x),\ldots,x_n\to f_n(x)$ is a triangular automorphism.  

Let us give three examples illustrating the concepts discussed above. In all cases we realize a local algebra $A$ as a factor of the polynomial algebra $\KK[S_1,\ldots,S_m]$, and denote by
$s_i$ the image of $S_i$ in $A$. 

\begin{example}
Take a local algebra $A=\KK[S_1,\ldots,S_n]/(S_iS_j, 1\le i,j\le n)$. We have 
$$
\exp(x_1s_1+\ldots x_ns_n)=1+x_1s_1+\ldots x_ns_n.
$$
It corresponds to an additive action on $\PP^n$ given by
$$
(y_1,\ldots,y_n)\circ [z_0:z_1:\ldots :z_n]=[z_0:z_1+y_1z_0:\ldots:z_n+y_nz_0].  
$$
We call this additive action standard. Here the basic polynomials are $f_1=x_1,\ldots,f_n=x_n$. 
\end{example} 

\begin{example}
Take $A=\KK[S_1]/(S_1^3)$ with the basis $s_1, s_2=s_1^2$ in $\mm$. We have
$$
\exp(x_1s_1+x_2s_2)=1+x_1s_1+(x_2+\frac{x_1^2}{2})s_2. 
$$
It corresponds to an additive action on $\PP^2$ given by
$$
(y_1,y_2)*[z_0:z_1:z_2]=[z_0:z_1+y_1z_0:z_2+y_1z_1+(y_2+\frac{y_1^2}{2})z_0].  
$$
In this case the basic polynomials are $f_1=x_1$ and $f_2=x_2+\frac{x_1^2}{2}$.
\end{example}

\begin{example}
Take a local algebra $A=\KK[S_1,S_2]/(S_1S_2, S_1^3-S_2^2)$ with the basis 
$$
s_1, \quad s_2, \quad s_3=s_1^2, \quad s_4=s_1^3=s_2^2
$$ 
in $\mm$. We have
$$
\exp(x_1s_1+x_2s_2+x_3s_3+x_4s_4)=1+x_1s_1+x_2s_2+(x_3+\frac{x_1^2}{2})s_3+(x_4+x_1x_3+\frac{x_2^2}{2}+\frac{x_1^3}{6})s_4. 
$$
It corresponds to an additive action on $\PP^4$ given by
$$
(y_1,y_2,y_3,y_4)*[z_0:z_1:z_2:z_3:z_4]=[z_0:z_1+y_1z_0:z_2+y_2z_0:
$$
$$
:z_3+y_1z_1+(y_3+\frac{y_1^2}{2})z_0:z_4+y_1z_3+y_2z_2+(y_3+\frac{y_1^2}{2})z_1+(y_4+y_1y_3+\frac{y_1^2}{2}+\frac{y_1^3}{6})z_0].  
$$
In this case the basic polynomials are 
$$
f_1=x_1,\quad f_2=x_2, \quad f_3= x_3+\frac{x_1^2}{2}, \quad f_4=x_4+x_1x_3+\frac{x_2^2}{2}+\frac{x_1^3}{6}. 
$$
\end{example}

It is shown in~\cite[Lemma~1.40]{AZ} that the linear span $V=\langle 1,f_1,\ldots,f_n\rangle$ is an $(n+1)$-dimensional subspace that generates the algebra $\KK[x_1,\ldots,x_n]$ and is invariant under all translations $(x_1+c_1,\ldots,x_n+c_n)$, $c_i\in\KK$. Subspaces with these properties are called basic subspaces of the algebra $\KK[x_1,\ldots,x_n]$. It is proved in~\cite[Theorem~2.14]{HT} that $(n+1)$-dimensional local algebras are in bijection with basic subspaces in $\KK[x_1,\ldots,x_n]$; see also \cite[Theorem~1.48]{AZ}. In particular, the number of equivalence classes of basic subspaces in
$\KK[x_1,\ldots,x_n]$ is indicated in the table in~(1).


\section{The main result}
\label{sec4}

We are ready to formulate the main result of this note. 

\begin{theorem}\label{tmain}
The additive action on $\PP^n$ corresponding to a local algebra $A$ is conjugate to the standard additive action by the automorphism given by basic polynomials of the algebra~$A$.
\end{theorem}

\begin{proof}
Let ${\bf x}=x_1s_1+\ldots+x_ns_n$ and ${\bf y}=y_1s_1+\ldots+y_ns_n$ be two elements of the maximal ideal $\mm$ of the algebra~$A$. 
By definition of basic polynomials, we have 
$$
\exp({\bf x})=1+f_1(x)s_1+\ldots+f_n(x)s_n.
$$ 
Let us consider the map $\varphi\colon1+\mm \to 1+\mm$ given by $\varphi(1+{\bf x})=\exp({\bf x})$ as an automorphism of the affine space $1+\mm$. The inverse map is given by $\varphi^{-1}(1+{\bf x})=1+\ln(1+{\bf x})$. 

Let 
$$
\GG_a^n\times\PP^n\to\PP^n, \quad (y,[z_0,\ldots,z_n])\mapsto y\circ [z_0:\ldots:z_n]
$$ 
be the standard additive action on $\PP^n$ and 
$$
\GG_a^n\times\PP^n\to\PP^n, \quad (y,[z_0,\ldots,z_n])\mapsto y*[z_0:\ldots:z_n]
$$ 
be the additive action on $\PP^n$ corresponding to the local algebra~$A$. 

Clearly, an additive action on $\PP(A)$ is uniquely determined by its restriction to the invariant affine chart $1+\mm$, which is the open orbit of this action.

\smallskip

So the claim of Theorem~\ref{tmain} can be reformulated in the following form. 

\begin{lemma} \label{lll}
In the notation introduced above, we have 
$$
\varphi(y\circ(1+{\bf x}))=y*\varphi(1+{\bf x})
$$ 
for all $y\in\GG_a^n$ and ${\bf x}\in\mm$. 
\end{lemma}  

\begin{proof}
By definition of the standard additive action, we have $y\circ(1+{\bf x})=1+{\bf x}+{\bf y}$. So, we obtain 
$$
\varphi(y\circ(1+{\bf x}))=\varphi(1+{\bf x}+{\bf y})=\exp({\bf x}+{\bf y})=\exp({\bf y})\exp({\bf x})=y*\varphi(1+{\bf x}).      
$$
\end{proof}
The proof of Theorem~\ref{tmain} is completed.
\end{proof} 

\begin{remark}
It is easy to see that the group $A^{\times}$ of invertible elements of a local algebra $A$ with $\dim A=n+1$ is isomorphic to $\GG_m\times\GG_a^n$. At the same time, for non-isomorphic
local algebras $(A,+,\circ)$ and $(A',+,*)$ we can not establish a polynomial bijection $\varphi\colon A\to A'$ such that $\varphi(a\circ b)=\varphi(a)*\varphi(b)$ for all $a,b\in A$. Indeed, it is proved in \cite[Lemma~7]{ABZ} that for any two non-isomorphic finite-dimensional associative (not necessarily commutative) unital algebras $A$ and $A'$ their multiplicative monoids
$(A,\circ)$ and $(A,*)$ are not isomorphic as affine algebraic monoids. 
\end{remark} 

Summing up, Theorem~\ref{tmain} shows that one $\Aut(U)$-conjugacy class of $\GG_a^n$-subgroups in $\Bir(\PP^n)$ corresponding to additive actions on $\PP^n$ intersects
the set of subgroups of the group ${\PGL(n+1)}$ in several ${\PGL(n+1)}$-conjugacy classes, and the number of such classes is indicated in the table in~(1). 


\end{document}